\documentclass[11pt,preqno,a4paper]{amsart}

\usepackage{enumitem}
\usepackage{amssymb,amsmath,mathtools}
\usepackage{hyperref}
\usepackage[nobysame]{amsrefs}
\usepackage{xcolor}
\usepackage{tikz}
\usepackage{dsfont}
\usetikzlibrary{angles, quotes, arrows.meta, positioning, calc}

\usepackage{microtype}
\usepackage{bbm}
\usepackage{a4wide}
\newcommand*{\mailto}[1]{\href{mailto:#1}{\nolinkurl{#1}}}

 \allowdisplaybreaks

\newtheorem{theorem}{Theorem}[section]

\newtheorem{lemma}[theorem]{Lemma}
\newtheorem{proposition}[theorem]{Proposition}

\newtheorem{remark}[theorem]{Remark}

\DeclareMathOperator{\supp}{supp}

\newcommand{\be}{\begin{equation}}
\newcommand{\ee}{\end{equation}}

\newcommand{\al}{\alpha}

\newcommand{\eq}{\begin{equation}}
\newcommand{\eeq}{\end{equation}}

\numberwithin{equation}{section}

\makeatletter
\newcommand{\dlmf}[1]{%
	 \cite[%
	\def\nextitem{\def\nextitem{, }}%
	\@for \el\coloneqq#1\do{\nextitem\href{http://dlmf.nist.gov/\el}{(\el)}}%
	]{dlmf}%
}
\makeatother

\begin{document}

\allowdisplaybreaks

\title[Robust Observability on Hyperbolic Surfaces]{Robust Observability for Schr\"odinger Equations with Singular Potentials on 2D Compact Hyperbolic Surfaces}

\author{Jiayan Wu, Ting Zhang and Ruze Zhou}
\address{School of mathematics, South China University of Technology
\\
Guangzhou 510641, China
}
\address{School of mathematics, Zhejiang University
\\
Hangzhou 310058, China
}
\address{School of mathematics, Zhejiang University
\\
Hangzhou 310058, China
}
\email{wujiayan@scut.edu.cn;  zhangting79@zju.edu.cn; 12435042@zju.edu.cn}

\date{\today}
\subjclass[2010]{35Q41, 35B45, 58J50, 93B07}

\keywords{Observability for PDE; Schr\"odinger operator; hyperbolic surface; rough potential; semiclassical analysis; spectral cluster estimate}

\begin{abstract}
This paper investigates the robustness of quantum observability on compact hyperbolic surfaces subject to $L^q$ potential perturbations with $2<q\leq \infty$. Since $L^q$ regularity is strictly subcritical in dimension two, for every non-empty open set $\Omega\subset M$ and $T>0$, the solution of $(i\partial_t+\Delta-V)u=0$ satisfies the space-time observability estimate
$$
\|u_0\|_{L^2(M)}^2
\leq C\int_0^T
\|e^{-it(-\Delta+V)}u_0\|_{L^2(\Omega)}^2\,dt.
$$
A cornerstone of our analysis is Sogge's spectral cluster estimate, which allows us to obtain refined spectrally localized semiclassical Strichartz estimates adapted to rough potentials and to show that the potential's contribution vanishes in the propagation of semiclassical measures. The full-support property of invariant semiclassical measures on hyperbolic surfaces associated with the eigenfunctions then yields the desired observability by contradiction. By the Hilbert Uniqueness Method, the corresponding internal controllability result follows.
\end{abstract}

\maketitle

\vspace{2em} 
\tableofcontents

\section{Introduction}

\subsection{Background}
The mathematical study of observability and controllability for the Schr\"odinger equation originates from fundamental challenges in the control of quantum systems. Historically, the study of observability originated from the problem of quantum state reconstruction (or quantum tomography), where one aims to uniquely determine the initial wave function $u_0$ from partial measurements collected over a spatio-temporal window $\Omega\times[0,T]$ \cite{Artiles2005, Leonhardt1997}. In turn, controllability concerns the ability of an external control, represented by a source or a potential, to steer a quantum system toward a prescribed target state—a task of fundamental importance in laser-induced chemical reactions and the implementation of quantum logic gates \cite{Peirce1988, DAlessandro2007}. From a dynamical perspective, the observability of quantum waves is intimately related to the underlying classical dynamics. In chaotic systems, such as those characterized by Anosov geodesic flows on hyperbolic surfaces, the exponential divergence of trajectories facilitates the rapid spreading of information, leading to the delocalization of eigenfunctions and the enhancement of observability, even in the presence of rough potentials \cite{Gutzwiller1990}. 

The rigorous mathematical study of controllability and observability was pioneered in the 1960s by Kalman \cite{Kalman1960}, who established the fundamental duality between these two concepts for finite-dimensional systems. This framework was subsequently extended to infinite-dimensional settings, such as distributed parameter systems governed by partial differential equations (PDEs), through the seminal works of Fattorini \cite{Fattorini1968} and Russell \cite{Russell1978}. In the 1980s, Lions \cite{LJL1,LJL2} introduced the Hilbert Uniqueness Method (HUM), which provided a systematic functional analytic approach to exact controllability by reducing it to the establishment of appropriate observability inequalities. Since then, the observability of Schr\"odinger-type equations has become a central theme, particularly regarding the influence of the underlying geometry and potential regularity.

From a dynamical systems perspective, the internal controllability of quantum waves is intimately related to the underlying classical geodesic flow. On hyperbolic surfaces, the chaotic nature of the Anosov flow promotes the rapid spreading of eigenfunctions, a phenomenon captured by the Quantum Unique Ergodicity (QUE) conjecture. 

Historically, the simplest geometric setting for such investigations has been the flat torus, where the arithmetic structure of the spectrum allows for precise Fourier-based analysis. Existing results on control and observability inequalities of the form \eqref{controlineq} have been obtained on flat tori, typically $$\mathbb{T}^n=\mathbb{R}^n/\mathbb{Z}^n.$$The first controllability result for $-\Delta$ was proved by Jaffard \cite{jaffard} and Haraux \cite{Haraux} in two dimensions and by Komornik \cite{Komornik} in higher dimensions. It is worth to note that the observational conclusions involving potential energy are also physically meaningful, see \cite{CTR,BZcontrol,AMcon,BBZcon,JLcon}. Then, it is natural to extend the problem to the Schr\"odinger operator $-\Delta+V$. For $V\in C^\infty(\mathbb{T}^2)$, the observability inequality was proved by Burq and Zworski \cite{BZcontrol} and they also conjectured that the observability should hold for $V\in L^\infty(\mathbb{T}^2)$. In the higher dimensions, Anantharaman and Maci\`a \cite[Theorem 4]{AMcon} obtained a relevant result. They proved that for the potential $V$ on $\mathbb{T}^d$ that is continuous almost everywhere and lies in $L^\infty$, which includes continuous potentials, the observability inequality \eqref{controlineq} holds. In the case of the rough potential $V\in L^2$, Bourgain, Burq and Zworski \cite{BBZcon} have established the inequality, and it is this work that inspires the present paper. In addition, they also prove that for functions satisfying the Floquet boundary condition and for $W\in L^p(\mathbb{T}^1),\ p>1$, the observability result will hold in one dimension \cite[Section 3]{BBZcon}. A very recent work \cite{Burq_Germain_Sorella_Zhu_2026} has shown that if the measure $\mu$ on $\mathbb{T}^d$ satisfies certain conditions, then the corresponding observability holds as well.

Observability inequalities for the Schr\"odinger operator $-\Delta+V$ on general manifolds $M$ are also of significant interest. The pioneering work in this area was established by Bardos, Lebeau and Rauch \cite{LG,BLR}. They proved the following geometric control condition:\begin{equation}
    \text{There exists $L=L(M,\Omega)>0$ such that every geodesic of length $L$ intersects $\Omega$,} 
\end{equation}which is necessary for proving the observability inequality when the geodesic flow is periodic. For the case $V\equiv 0$ on a hyperbolic surface of constant negative curvature, the observability was obtained by Dyatlov and Jin \cite{DJ2018,JLcon}. There are also many other results related to observability and control. We refer the reader to \cite{AR,ANM,BZgeometric,TW,HWW,LJL1,LJL2,CB,RTTT,Miller,DM,WWZ,burqzhu,FGW} for further details.


Despite the progress described above,  much less is known for   Schr\"odinger operator with low-regularity potentials on negatively curved manifolds. In particular, it remains unclear to what extent the observability mechanism on hyperbolic surfaces is stable under rough perturbations. Our work is motivated by the observability theorem of Jin \cite{JLcon} for the Schr\"odinger equation on compact hyperbolic surfaces, which relies on the full-support property of semiclassical measures proved by Dyatlov and Jin \cite{DJ2018}. While these results concern the unperturbed Laplacian, it is natural to ask whether the same observability mechanism remains valid for Schr\"odinger operators with low-regularity potentials. The purpose of this paper is to address this question for real-valued potentials $V\in L^\frac{p}{p-2}(M),\, 2\leq p< 4$, where the potential condition holds through the paper.

\subsection{Applications}
The space-time observability inequality \eqref{controlineq} guarantees that for a known background rough potential $V$, the unknown initial state $u_0 \in L^2(M)$ can be stably and uniquely recovered from localized measurements on the subregion $\Omega \times [0,T]$.

Robust observability inequalities play a pivotal role in coefficient inverse problems, such as the determination of quantum potentials (see, e.g., the Bukhgeim–Klibanov framework \cite{KMV1} and related inverse potential stability results \cite{FO,FAO,FLL,KMSSS,ST,CarrilloMercadoMorales,KianSoccorsi,KianPhanSoccorsi}). On compact hyperbolic surfaces, classical Carleman-based inversion methods often encounter severe geometric obstructions due to non-trivial topology and negative curvature. By bridging the Ehrenfest time scale with spectrally localized Strichartz estimates, our work provides a robust analytical foundation that supports stable parameter identification and inversion algorithms in chaotic quantum systems.

\section{Main Results and Proof Strategy}

\subsection{Main result}Let $M$ be a compact connected hyperbolic surface of dimension $2$ with constant curvature $-1$, and let $\Delta_g$ be the Laplace-Beltrami operator on $M$. For the Riemannian manifold $(M,g)$,  $\Delta_g$ is locally given by \begin{equation}
    \Delta_gf=\frac{1}{\sqrt{|g|}}\sum_{i,j}\partial_i\left(\sqrt{|g|}g^{ij}\partial_jf\right),
\end{equation}where $(g^{ij})=g^{-1}$ denotes the inverse of the metric tensor and $|g|=\det(g_{ij})$ is the determinant of the metric matrix. The operator coincides with the usual Laplacian and we can see \cite[Chapter 3]{Riemgeoana} for further details. For simplicity, we shall denote $\Delta_g$ by $\Delta$ throughout the paper.

We consider the following Schr\"odinger equation\begin{equation}
    (i\partial_t+\Delta-V)u=0, 
\end{equation}
on $M$, where $V\in L^\frac{p}{p-2}(M),\, 2\leq p< 4$ is a real-valued potential. We denote by $e^{-it(-\Delta+V)}$   the unitary group generated by the self-adjoint operator $-\Delta+V$ with domain $H^2(M)$.
\begin{remark}
    It is worth noting that $p=2$ corresponding to the case $V\in L^\infty$ that has been stated to be easy to prove in \cite[Remark 1.4]{JLcon}. Thus, we only need to consider the case $2<p< 4$. And obviously the spectral of $-\Delta+V$ will be extended finitely to the negative half of the real axis. Then, we can reset the self-adjoint operator $-\Delta+V+C$ whose spectral is strictly positive. Hence, the unitary group $e^{-it(-\Delta+V)}$ naturally transformed into $e^{-it(-\Delta+V+C)}=e^{-itC}e^{-it(-\Delta+V)}$ which merely has an additional constant coefficient that does not work under the norm compared to $e^{-it(-\Delta+V)}$. Thus, for the sake of simplicity, we set the spectral of $-\Delta+V$ is positive in the thorough paper.
\end{remark}

The main purpose of this work is to establish the following space‑time observability inequality.\begin{theorem}\label{obser}
    Let $\Omega\subset M$ be a non-empty open set and  $T>0$. Then there exists a constant $K>0$, depending only on $M,\ \Omega,\ V$ and $T$, such that for any $u_0\in L^2(M)$, we have \begin{equation}\label{controlineq}
        \|u_0\|_{L^2(M)}^2\leq K\int_0^T\|e^{-it(-\Delta+V)}u_0\|^2_{L^2(\Omega)}\,dt.
    \end{equation}
\end{theorem}

We can also obtain the following stationary observability inequality. \begin{theorem}\label{staobser}
    Assume $a\in C_c^\infty(T^*M)$ satisfies $a|_{S^*M}\not\equiv 0$. Then there exist constants $C,\ h_0>0$ depending on $M,\ a,\ V$, such that for any $0<h<h_0,\ u\in H^2(M)$, \begin{equation}\label{staineq}
        \|u\|_{L^2}\leq   C\left(\| \mathrm{Op}_h(a) u \|_{L^2} + \frac{\log(1/h)}{h} \| (-h^2\Delta+h^2V-1) u \|_{L^2}\right).
    \end{equation}
\end{theorem}

The two observational results in Theorem~\ref{obser} and Theorem~\ref{staobser} are closely related. Within the framework of \cite{JLcon}, it is natural to expect that Theorem~\ref{obser} could be recovered from the stationary estimate in Theorem~\ref{staobser} via a standard transference principle. Although such a derivation appears feasible, we focus here on providing another direct approach. By employing $L^p$ spectral clusters estimates for the stationary case and spectrally localized semiclassical Strichartz estimates for the dynamic case, we aim to highlight the intrinsic consistency of high-energy delocalization in both settings. Once we prove Theorem~\ref{obser}, using the HUM method \cite{LJL1,LJL2}, the following control result is valid.
\begin{theorem}
    Let $\Omega\subset M$ be a non-empty open set and let $T>0$. For any $u_0\in L^2(M)$, there exists $f\in L^2([0,T]\times\Omega)$ such that the solution of the equation\begin{equation}
        \begin{cases}
            \,(i\partial_t+\Delta-V(x))u(t,x)=f \mathds{1}_{[0,T]\times\Omega}(t,x) \\
            \,u(0,\cdot)=u_0
        \end{cases}
    \end{equation}satisfies \begin{equation}
        u(T,\cdot)\equiv0.
    \end{equation}
\end{theorem}

\begin{remark} \label{rem:generalizations}
We briefly discuss several potential generalizations of our observability results (Theorem~\ref{obser} and~\ref{staobser}). First, it is obviously that weather the observability results for $V\in L^2$ are valid is still unknown. To prove the conclusion, we need to refine the estimates used during our statement or explore new methods. And while our proof heavily relies on the assumption of the constant curvature $K=-1$ in order to ensure the full support of the measure \cite{DJ2018}, it is not difficult to prove the observability conclusions with variable negative curvatures, since the full-support property proved by Dyatlov, Jin and Nonnenmacher \cite{DJN}. Then, the extension of our results to compact hyperbolic manifolds of dimension $d \ge 3$ constitutes another promising direction. However, our argument does not directly apply because it heavily relies on two-dimensional geometric properties. Therefore, we need to develop a new method. Our dynamic observability holds for extremely rough potentials $V \in L^p(M),\, 2< p\leq \infty$ in two dimensions. For higher-dimensional settings, one might consider adjusting the regularity class to $V \in L^p(M)$ with $p \ge d/2$, where similar Strichartz and resolvent estimates could be established within our chaotic spectral framework. We note that the endpoint case $p=d/2$ is scaling-critical and typically involves additional technical nuances compared to the subcritical case $p>d/2$. Finally, we can also consider our system with the time-related potential or try to prove the observability of the Laplacian on the fractal subsets. However, under these assumptions, our method may fail and require a different approach.
\end{remark}

\subsection{Innovation and strategy}

The main contribution of this work is to establish observability inequalities for the Schr\"odinger operator on a compact hyperbolic surface $M$ with a rough potential $V\in L^\frac{p}{p-2}(M)$. This result can be viewed from two complementary perspectives.

First, compared with the work of Bourgain–Burq–Zworski \cite{BBZcon} on two-dimensional flat tori with rough potentials $V\in L^2$ , our setting replaces the flat torus by a compact hyperbolic surface, where the geodesic flow is Anosov and no global Fourier lattice structure is available precluding the use of classical discrete harmonic analysis methods. The main difficulty is therefore to combine rough-potential perturbation arguments with the microlocal and dynamical features of negatively curved geometry.

Second, compared with the results of Dyatlov–Jin \cite{DJ2018} and Jin \cite{JLcon} on compact hyperbolic surfaces, which treat the free Schr\"odinger equation and rely on the full-support property of semiclassical measures, the present work allows an $L^\frac{p}{p-2}$-rough potential. Note that in dimension two, $V\in L^\frac{p}{p-2}(M),\, 2\leq p< 4$ is strictly subcritical with respect to the scaling of the Laplacian, as the critical regularity is $V\in L^1(M)$. This subcriticality is crucial for our analysis, as it allows the potential to be treated as a perturbation of the semiclassical Laplacian with a quantitative power gain in $h$. In this sense, our theorem extends the   observability theory on hyperbolic surfaces from the unperturbed Laplacian to the Schr\"odinger operator $-\Delta+V$ with low-regularity potentials. The corresponding controllability result follows by the Hilbert Uniqueness Method.

To achieve this, we develop an auxiliary semiclassical estimate that is of independent interest. We establish a spectrally localized semiclassical Strichartz estimate (Theorem~\ref{semistrichartz}) tailored to the $L^\frac{p}{p-2}$ -regularity of $V$, which bridges the gap between high-energy eigenfunction behavior and the dynamical stability of the Schr\"odinger flow.

Finally, we provide a physical heuristic for the role of rough potentials. From a physical perspective, the inclusion of a rough potential $V\in L^\frac{p}{p-2}(M)$ introduces a form of microscopic scattering. In integrable systems, quantum scarring or concentration of eigenfunctions (e.g., along periodic orbits) arises from the constructive interference of waves. On hyperbolic surfaces, the underlying geodesic flow is already Anosov (chaotic), which promotes delocalization. Our results show that an $L^\frac{p}{p-2}$ potential, despite its low regularity, is subcritical enough that it cannot create new localized trapping centers or scars. Instead, the high-energy limit remains governed by the ergodic properties of the free Hamiltonian. This indicates that in the propagation of semiclassical measures, the contribution of the potential term vanishes in the semiclassical limit; see Section~\ref{proofofsta}.




The paper proceeds in the following steps. Firstly, we translate the spectral bounds into a resolvent estimate in $L^{p'}\to L^p$(Proposition~\ref{resolwithV}), via the $TT^*$ argument and spectral clusters estimate established by Sogge. Using the semiclassical Fourier transforms, we then derive the spectrally localized semiclassical Strichartz estimates, which provide the necessary control over the time-dependent evolution under the $L^\frac{p}{p-2}$ potential. Next, we analyze the semiclassical defect measure $\mu$ associated with a sequence of solutions. We demonstrate that $\mu$ is invariant under the Hamiltonian flow (i.e., the geodesic flow). In the hyperbolic setting, the ergodicity and full-support property of such measures (following Dyatlov-Jin \cite{DJ2018}) ensure that $\mu$ cannot vanish on any non-empty open set $\Omega$. Finally,  we argue by contradiction and construct a measure that vanishes on the observation region. The invariance and full support of the measure then lead to a contradiction for high-frequency terms. The low-frequency regime is handled via the unique continuation property (UCP), completing the proof.

\subsection{Notation}

Throughout the paper, we write $A\lesssim B$ to denote $A\leq CB$, where the implicit constant $C$, which may vary from line to line, is generic and depends on $M,\Omega$ and the fixed time $T$, but not on the semiclassical parameter $h$. We also assume that there exists a constant $h_0>0$ such that $0<h<h_0$.

We define the spectral projection operator by \begin{equation}
    \Pi_{[a,b]}(P)f=\sum_{\lambda_j\in [a,b]}\langle f,\psi_j\rangle\psi_j,\quad \mbox{for } f\in L^2(M),
\end{equation}
which projects onto the frequency interval $[a,b]$, where $\{\psi_j\}_{j=0}^\infty$ denotes the orthogonal basis of the eigenfunctions of $P$ corresponding to the eigenvalue $\lambda_j^2$, respectively.

\section{Semiclassical Analysis and Refined Spectral Estimates}\label{prelim}

\subsection{Semiclassical analysis and the defect measure}\label{semidef}

In this subsection, we recall some definitions and facts from semiclassical analysis that will serve as the core framework throughout the proof. For the convenience of semiclassical analysis, we shall henceforth adopt the parameter $h=\lambda^{-1}\to 0$.

We now introduce the cotangent space $T^*M$. For each $x\in M$, the collections of all the tangent vectors at the point $x$ constitutes the tangent space $T_xM$ whose dual is denoted by $T^*_xM$. The cotangent bundle $T^*M$ is defined as \begin{equation}
    T^*M=\bigsqcup_{x\in M}T^*_xM,
\end{equation}and the elements $(x,\xi)\in T^*M$ carries both the position information and the frequency information at $x$. And $S^*M=\{(x,\xi)\in T^*M\ |\  |\xi|_g=1\}$ denotes the unit cotangent sphere bundle.

Next, we introduce the semiclassical quantization operator on the manifold. To this end, let us first recall the definition of the symbol class. For $m\in \mathbb{R}$ and for any pair of multi-index $\alpha,\beta$, $a\in C^\infty(\mathbb{R}^n\times\mathbb{R}^n)$ belongs to the symbol class $S^m$, if there exists a constant $C_{\alpha,\beta}$ such that \begin{equation}
    |\partial_x^\alpha\partial_\xi^\beta a(x,\xi)|\leq C_{\alpha,\beta}\langle\xi\rangle^{m-|\beta|},
\end{equation}where $\langle\xi\rangle=(1+|\xi|^2)^{\frac{1}{2}}$. For a symbol $a\in S^m (\mathbb{R}^{2n})$, we denote $\mathrm{Op}_h(a)$ as its semiclassical quantization, defined by\begin{equation}
    (\mathrm{Op}_h(a)u)(x) = \frac{1}{(2\pi h)^n} \iint e^{\frac{i}{h}(x-y)\cdot\xi} a(x, \xi) u(y)  dy d\xi .
\end{equation}This definition is extended to the manifold $M$ via a partition of unity together with the compatibility of local coordinate charts. We can refer to \cite[Chapter 4]{SemianaZM} for more details.

Then, we define the notion of the semiclassical defect measure. Let $u_j$ be a family of eigenfunctions of $-\Delta$ with eigenvalues $h_j^{-2}\to \infty$, that is \begin{equation}
    (-h_j^2\Delta-1)u_j=0, \ \|u_j\|_{L^2}=1, \ h_j>0, \ h_j\to 0 \ as \ j\to \infty.
\end{equation}We say that $u_j$ converge semiclassically to probability measure $\mu$ on $T^*M$, which is called the semiclassical defect measure \cite[Theorem 5.2]{SemianaZM}, if \begin{equation}
    \langle\mathrm{Op}_{h_j}(a)u_j,u_j\rangle\to \int_{T^*M}ad\mu \ as \ j\to \infty,\ for \ all \ a\in C_c^\infty(T^*M).
\end{equation}
By standard results in semiclassical analysis (see for instance \cite[Chapter 5]{SemianaZM}), any such limit measure $\mu$ is a Radon probability measure. Furthermore, the semiclassical defect measure $\mu$ is supported on the unit cotangent bundle $S^*M$ and is invariant along the geodesic flow.

Given $f,g\in C^\infty (\mathbb{R}^{2n})$, we define Hamiltonian vector fields which generate the dynamics, \begin{equation}
    H_f=\langle\partial_\xi f,\partial_x\rangle-\langle\partial_x f,\partial_\xi\rangle.
\end{equation}Thus, Poisson bracket can be defined as \begin{equation}
    \{f,g\}:=H_fg.
\end{equation}We can refer to \cite[Section 2.4]{SemianaZM}. For $p(x,\xi)=|\xi|^2_g-1$, we have that Hamiltonian vector fields generate the geodesic flow $\Phi_t$.

The semiclassical Fourier transform and inverse semiclassical Fourier transform in time is given by \begin{equation}
    \begin{aligned}
        \mathcal{F}_h[w(t)]=\hat{w}_h(\tau)=\frac{1}{2\pi h}\int e^{-it\tau/h}w(t)dt, \\
    \mathcal{F}^{-1}_h[w(t)]=\check{w}_h(\tau)=\int e^{it\tau/h}w(t)dt.
    \end{aligned}
\end{equation}
For the sake of simplicity in notation, we set $P_V=-h^2\Delta+h^2V$ and $P_V^n=-h_n^2\Delta+h_n^2V$.

\subsection{Spectral cluster estimate and resolvent estimate}

To proceed toward the result, we first consider such a eigenvalue equation 
\begin{align}\label{characteristicequ}
    (-h^2\Delta-1)u=0
\end{align} 
and give the spectral cluster estimates of the eigenfunctions.

To better understand the significance of our spectral cluster estimates, we recall the classical $L^p$ estimates for spectral clusters on general $n$-dimensional compact Riemannian manifolds established by Sogge \cite[Proposition 2.1]{sogge} which will be used during the proof. 
\begin{theorem}\label{sogge}
    For a general $n$-dimensional smooth compact manifold without boundary $M$, $n\geq 2$ the universal estimate states that
\begin{equation} \label{eq:sogge}
\|\Pi_{[\lambda, \lambda+1]}(-\Delta) u\|_{L^p(M)} \lesssim \lambda^{\delta(p)} \|\Pi_{[\lambda, \lambda+1]}(-\Delta) u\|_{L^2(M)},
\end{equation}
where the optimal exponent is given by
\begin{equation}\label{soggeexp}
        \delta(p)=\begin{cases}
            \frac{n-1}{2}(\frac{1}{2}-\frac{1}{p}), \ &2\leq p\leq \frac{2(n+1)}{n-1} \\
            n(\frac{1}{2}-\frac{1}{p})-\frac{1}{2}, \ & \frac{2(n+1)}{n-1}\leq p\leq \infty
        \end{cases}.
    \end{equation}
\end{theorem}

\begin{remark}
    In the following proof, we need $1-4\delta(p)>0$. Since \eqref{soggeexp}, we can get $p<8$. The stronger restriction $2\leq p< 4$ ensures the perturbation $Vu$ remains well-defined in $L^2$. Thus, our assumption is rational.
\end{remark}


Beyond the global estimates on the full manifold, another central theme in spectral geometry is the behavior of eigenfunctions restricted to submanifolds. In their landmark paper, Burq, G\'{e}rard, and Tzvetkov \cite{BGT} initiated the study of restriction estimates for Laplace--Beltrami eigenfunctions to submanifolds. In particular, for a compact Riemannian surface $M$ and a smooth curve $\gamma \subset M$,  Burq, G\'{e}rard, and Tzvetkov proved that the restriction of an eigenfunction $e_\lambda$ (with $-\Delta e_\lambda = \lambda^2 e_\lambda$) satisfies
\begin{equation}
\|e_\lambda\|_{L^q(\gamma)} \lesssim (1+\lambda)^{\delta_1(q)} \|e_\lambda\|_{L^2(M)},
\end{equation}
where $$\delta_1(q)=\begin{cases}
    \frac{1}{4},\ &2\leq q\leq 4, \\
    \frac{1}{2}-\frac{1}{q},\ &4\leq q\leq \infty,
\end{cases}$$ which is sharp on the standard sphere.

More recently, these restriction estimates have been generalized to rougher geometric settings, such as restrictions to fractal subsets or Borel measures. Following the pioneer work of Eswarathasan and Pramanik \cite{EP22} on fractal restrictions, Gao, Miao, and Xi \cite{GMX} established highly refined $L^p$ restriction estimates for eigenfunctions on general Riemannian surfaces to fractal subsets of Hausdorff dimension $\alpha \in (0,2]$. Their work beautifully unifies and refines both Sogge's global estimates (corresponding to $\alpha=2$) and BGT's curve restrictions (corresponding to $\alpha=1$). Let $\mu$ be an $\al$-dimensional probability measure on $M$. They have proved that \begin{equation}
    \|e_\lambda\|_{L^p(M;d\mu)} \lesssim \lambda^{\delta(\al,p)+\epsilon} \|e_\lambda\|_{L^2(M)}
\end{equation}where $$\delta(\al,p)=\begin{cases}
    \frac{1}{2}-\frac{\al}{p},\ &0<\al\leq \frac{1}{2},\ 2\leq p\leq \infty, \\
    \frac{1}{4},\ & \frac{1}{2}<\al\leq 1,\ 2\leq p\leq 4\al, \\
    \frac{1}{2}-\frac{\al}{p},\ &\frac{1}{2}<\al\leq 1,\ 4\al< p\leq \infty, \\
    \frac{1}{4}-\frac{\al-1}{2p},\ &1<\al\leq 2,\ 2\leq p\leq 2\al+2, \\
    \frac{1}{2}-\frac{\al}{p},\ &1<\al\leq 2,\ 2\al+2< p\leq \infty. \\
\end{cases}$$
Due to the worst-case scenarios on spherical-like geometries, their general bounds for $\alpha=2$ still carry the classical Sogge exponent of $\delta(2,4) = \frac{1}{8}$.

Also, it is worth to note that Sogge's estimate Theorem \ref{sogge} is sharp for the standard sphere $\mathbb{S}^2$ due to the extreme concentration of zonal harmonics around the poles \cite{sogge,kakeya}. However, it is not a sharp estimate in the presence of nonpositive curvature and has logarithmic improvements \cite{BSlog} \begin{equation}
    \|\Pi_{[\lambda,\lambda+(\log\lambda)^{-1}]}(-\Delta)\|_{L^2\to L^{p_c}}\lesssim\frac{\lambda^{1/p_c}}{(\log\lambda)^{\epsilon_0}}
\end{equation}
or even further improvements \cite{imsogge,SZL4,SZ2}\begin{equation}
    \|\phi_j\|_{L^4}\lesssim o(\lambda_j^{1/8})
\end{equation}
and \begin{equation}
     \|\Pi_{[\lambda,\lambda+(\log\lambda)^{-1}]}(-\Delta)\|_{L^2\to L^{p_c}}\lesssim\lambda^{1/p_c}(\log\log\lambda)^{-\frac{2}{(n+1)^2}},
\end{equation} where $p_c=\frac{2(n+1)}{n-1}$ and $\epsilon_0>0$. 

For the Schr\"odinger operator $H_V=-\Delta+V$ with singular potential, which corresponds to the operator $P_V=h^2H_V$ with $h=\lambda^{-1}$ in the semiclassical setting, we also have the similar bound as Theorem \ref{sogge} to the solutions of the characteristic equation $$(-\Delta+V)u=\lambda^2u.$$ Before presenting the conclusion, we first define the Kato class. For all $V\in L^1(M)$ satisfying $$\lim_{\delta\to 0}\sup_{x\in M}\int_{d_g(x,y)<\delta} |V(y)|W_n(d_g(x,y))dy=0$$where \begin{equation}
    W_n(r)=\begin{cases}
        r^{2-n},\, &n\geq 3, \\
        \log(2+r^{-1}),\, &n=2,
    \end{cases}
\end{equation}we claim $V\in \mathcal{K}(M)$. By the H\"older inequality, we have $L^{q}(M)\subset\mathcal{K}(M)\subset L^1(M)$ for all $q > \frac{n}{2}$.
\begin{theorem}[Theorem 1.1 in \cite{HWZ}]\label{sogge_disturb}
    Let $M$ be a compact manifold of dimension $n\geq 2$. If $V\in \mathcal{K}(M)\bigcap L^{\frac{n}{2}}(M)$, then we have for some constant $C_V$ independent of $\lambda$ \begin{equation}
        \|\Pi_{[\lambda,\lambda+1]}(H_V)u\|_{L^p}\leq C_V \lambda^{\delta(p)}\|\Pi_{[\lambda,\lambda+1]}(H_V)u\|_{L^2},
    \end{equation}where $\delta(p)$ is Sogge's exponent which we have stated in \eqref{soggeexp}.
\end{theorem}For simplicity, we denote $\Pi_{[\lambda,\lambda+1]}(H_V)$ by $\Pi_{[\lambda,\lambda+1]}$ and $h^{-1}=\lambda$, where $\lambda^2$ is the eigenvalue of $H_V$, throughout the paper.

Moreover, we have a reasonable conjecture that we can improve Theorem \ref{sogge} to the $\lambda^\epsilon$ bound, which is difficult to prove and need some innovative microlocal techniques. Nevertheless, in our work, the universal bound provided by Theorem \ref{sogge} and Theorem \ref{sogge_disturb} is sufficient.

We now turn to the proof of the following resolvent estimate.
\begin{proposition}\label{resolwithV}
    For any $f\in L^{p'}$, $|\eta|\geq h$, $2\leq p<4$, $0<\epsilon \ll 1$ and $V\in L^\frac{p}{p-2}$, we have \begin{equation}
    \|(-h^2\Delta+h^2V-1-i\eta)^{-1}f\|_{L^p}\lesssim h^{-1-2\delta(p)-\epsilon}\|f\|_{L^{p'}}.
\end{equation}
\end{proposition}
\begin{proof}
By a $TT^*$ argument and the identity 
$$
(-h^2\Delta+h^2V-1-i\eta)^{-1}=(-h^2\Delta+h^2V-1-i\eta)^{-\frac{1}{2}}((-h^2\Delta+h^2V-1+i\eta)^{-\frac{1}{2}})^*,
$$
it suffices to prove
\begin{equation}
    \|(-h^2\Delta+h^2V-1-i\eta)^{-\frac{1}{2}}f\|_{L^p}\lesssim h^{-1/2-\delta(p)-\epsilon}\|f\|_{L^2}.
\end{equation}

We set $P=-h^2\Delta+h^2V-1$ and  $E_j=h^2\lambda_j^2-1$, where $\lambda_j^2$ is the eigenvalue of $H_V=-\Delta+V$ corresponding to the eigenfunction $\phi_j$. For $f\in L^2$, we write 
$$
f=\sum_{j=0}^\infty\langle f,\phi_j\rangle\phi_j:=\sum_{j=0}^\infty f_j\phi_j.
$$
Then
$$
(P-i\eta)^{-\frac{1}{2}}f=\sum_{j=0}^\infty f_j(P-i\eta)^{-\frac{1}{2}}\phi_j=\sum_{j=0}^\infty f_j(E_j-i\eta)^{-\frac{1}{2}}\phi_j.
$$
We decompose 
\begin{equation}
    u:=(P-i\eta)^{-\frac{1}{2}}f=u_0+\sum_{k=0}^\infty(P-i\eta)^{-\frac{1}{2}}\Pi'_kf,
\end{equation}
where \begin{equation}
    \Pi'_kf=\sum_{|E_j|\in J_k}f_j\phi_j,
\end{equation} 
with $J_k=[2^{k-1}h,2^kh]$ and 
\begin{equation}
    u_0=\sum_{|E_j|<h/2}(P-i\eta)^{-\frac{1}{2}}f_j\phi_j.
\end{equation}

For the near resonance term $u_0$ and $|\eta|\geq h$,  Theorem~\ref{sogge_disturb} gives
\begin{equation}\label{nearbigh}
    \begin{aligned}
        \|u_0\|_{L^p}&\lesssim\|\sum_{|E_j|<h/2} f_j(E_j-i\eta)^{-\frac{1}{2}}\phi_j\|_{L^p}\lesssim h^{-\delta(p)}\|\sum_{\lambda_j\in [\frac{1}{h}-\frac{1}{2},\frac{1}{h}+\frac{1}{2}]} f_j(E_j-i\eta)^{-\frac{1}{2}}\phi_j\|_{L^2} \\
        &\lesssim |\eta|^{-\frac{1}{2}}h^{-\delta(p)}\|f\|_{L^2}\lesssim h^{-\frac{1}{2}-\delta(p)}\|f\|_{L^2}.
    \end{aligned}
\end{equation}

For the away from resonance term, applying Theorem~\ref{sogge_disturb} in the similar way and considering $E_j>0$ first, we obtain
\begin{equation}\label{farren}
    \begin{aligned}
        \|\sum_{k=0}^\infty(P-i\eta)^{-\frac{1}{2}}\Pi'_{k,+}f\|_{L^p} &\lesssim \sum_{k=0}^\infty (2^kh)^{-\frac{1}{2}}\|\Pi'_{k,+}f\|_{L^p} \\
        &\simeq\sum_{k=0}^\infty (2^kh)^{-\frac{1}{2}}\|\sum_{\lambda_j\in h^{-1}[\sqrt{2^{k-1}h+1},\sqrt{2^{k}h+1}]}f_j\phi_j\|_{L^p}\\
        &\lesssim \sum_{k=0}^\infty (2^kh)^{-\frac{1}{2}}h^{-\delta(p)}(\sqrt{2^{k}h+1})^{\delta(p)}(\frac{2^k}{\sqrt{2^{k}h+1}})^{1/2}\|\Pi'_kf\|_{L^2} \\
        &=h^{-\frac{1}{2}-\delta(p)}\sum_{k=0}^\infty (\sqrt{2^{k}h+1})^{-1/2+\delta(p)}\|\Pi'_kf\|_{L^2},
    \end{aligned}
\end{equation}where $-1/2+\delta(p)\leq -3/8<0$.

For $0\leq k\leq -\log_2h$, we have \begin{equation}\label{far1}
    \begin{aligned}
        \sum_{0\leq k\leq -\log_2h}(\sqrt{2^{k}h+1})^{-1/2+\delta(p)}\|\Pi'_kf\|_{L^2}&\leq (\sum_{0\leq k\leq -\log_2h}1)^{1/2}\|f\|_{L^2} \\
        &\lesssim (-\log_2h)^{1/2}\|f\|_{L^2}\lesssim h^{-\epsilon}\|f\|_{L^2}.
    \end{aligned}
\end{equation}For $k> -\log_2h$, we have \begin{equation}\label{far2}
    \begin{aligned}
        \sum_{k>-\log_2h}(\sqrt{2^{k}h+1})^{-1/2+\delta(p)}\|\Pi'_kf\|_{L^2}&\lesssim\sum_{k>-\log_2h}(\sqrt{2^{k}h+1})^{-3/8}\|\Pi'_kf\|_{L^2} \\
        &\lesssim (\sum_{k>-\log_2h}(\sqrt{2^{k}h+1})^{-3/4})^{1/2}\|f\|_{L^2} \\
        &\lesssim (\sum_{k=0}^\infty2^{-3k/8})^{1/2}\|f\|_{L^2}\lesssim\|f\|_{L^2}.
    \end{aligned}
\end{equation}

For $E_j<0$, we have $2^kh<1$, that is $0\leq k\leq -\log_2h$, which means we can estimate this part similarly to \eqref{far1}. When $2^kh\leq 1/2$, we obtain \begin{equation}\label{neg1}
    \begin{aligned}
        \|\sum_{k=0}^\infty(P-i\eta)^{-\frac{1}{2}}\Pi'_{k,-}f\|_{L^p} &\lesssim \sum_{k=0}^\infty (2^kh)^{-\frac{1}{2}}\|\Pi'_{k,-}f\|_{L^p} \\
        &\simeq\sum_{0\leq k\leq -\log_2h-1} (2^kh)^{-\frac{1}{2}}\|\sum_{\lambda_j\in h^{-1}[\sqrt{1-2^{k}h},\sqrt{1-2^{k-1}h}]}f_j\phi_j\|_{L^p}\\
        &\lesssim \sum_{0\leq k\leq -\log_2h-1} (2^kh)^{-\frac{1}{2}}h^{-\delta(p)}(\sqrt{1-2^{k}h})^{\delta(p)}(\frac{2^k}{\sqrt{1-2^{k}h}})^{1/2}\|\Pi'_kf\|_{L^2} \\
        &=h^{-\frac{1}{2}-\delta(p)}\sum_{0\leq k\leq -\log_2h-1} (\sqrt{1-2^{k}h})^{-1/2+\delta(p)}\|\Pi'_kf\|_{L^2} \\
        &\lesssim h^{-\frac{1}{2}-\delta(p)-\epsilon}\|f\|_{L^2}.
    \end{aligned}
\end{equation}

For $1/2<2^kh<1$, there is at most one $k$ satisfying the assumption. Then, for such $k$, we have \begin{equation}\label{neg2}
    \begin{aligned}
        &(2^kh)^{-\frac{1}{2}}\|\Pi'_kf\|_{L^p}\lesssim\sum_{m=0}^{[h^{-1}]+1}\|\Pi_{[m,m+1]}f\|_{L^p} \\
        \lesssim&h^{-\delta(p)}\sum_{m=0}^{[h^{-1}]+1}\|\Pi_{[m,m+1]}f\|_{L^2} \lesssim h^{-\delta(p)-1/2}\|f\|_{L^2}.
    \end{aligned}
\end{equation}

Combining \eqref{nearbigh}, \eqref{farren}, \eqref{far1}, \eqref{far2}, \eqref{neg1} and \eqref{neg2}, we complete the proof.
\end{proof}

\subsection{Unique continuation and low-frequency analysis}
To complement the high-frequency analysis conducted via semiclassical measures, we require the unique continuation property, as is shown by the following lemma, to handle the low-energy part of the spectrum.
\begin{lemma}[Strong Unique Continuation Property \cite{JKUCP}] \label{lem:ucp}
Let $M$ be a connected $2$-dimensional Riemannian manifold (not necessarily compact), and let $V \in L^p_{\mathrm{loc}}(M)$ for some $p > 1$. Suppose that $u \in H^2_{\mathrm{loc}}(M)$ satisfies the differential inequality
\begin{equation} \label{eq:gen-ucp-ineq}
    |\Delta u(x)| \leq |V(x) u(x)|
\end{equation}
almost everywhere on $M$. Then $u$ possesses the strong unique continuation property on $M$. In particular, if $u$ vanishes on a non-empty open subset $\Omega \subset M$, then $u \equiv 0$ identically on $M$.
\end{lemma}

\begin{proof}
     Since the strong unique continuation property is local, the global result on the connected manifold $M$ follows from the local uniqueness on coordinate charts. On any coordinate chart, the Laplace--Beltrami operator $\Delta$ is a second-order elliptic operator with smooth coefficients. Under local conformal coordinates, the local metric is given by $g = e^{2\phi}(dx_1^2 + dx_2^2)$ where $\phi$ is a smooth, real-valued function \cite[Page 246]{Riemgeoana}. Consequently, the differential inequality \eqref{eq:gen-ucp-ineq} on $M$ reduces to the flat one in $\mathbb{R}^2$. More precisely, we have \begin{equation}
        \begin{aligned}
            |\Delta u|=|e^{-2\phi}\Delta_{\mathbb{R}^2} u|\leq |V(x) u(x)| \\
            \implies|\Delta_{\mathbb{R}^2} u|\leq |e^{2\phi}V(x) u(x)|,
        \end{aligned}
    \end{equation}where $e^{2\phi}V\in L^p_{\mathrm{loc}}(M)$ for $p>1$. Therefore, the lemma is a direct consequence of the local Euclidean strong unique continuation property established by Jerison and Kenig \cite[Remark 6.7]{JKUCP}.
\end{proof}

\section{Proof of the Theorems}\label{finalproof}

\subsection{Proof of stationary observability}\label{proofofsta}


Initially, we state the following theorem, which demonstrates the full-support property of the measure related to the eigenfunctions of $-\Delta$.
\begin{theorem}[{\cite[Theorem 1]{DJ2018}}]\label{fullsupp}
    Let $\mu$ be a semiclassical defect measure. Then, $\supp \mu=S^*M$, that is for every non-empty open set $\mathcal{U}\subset S^*M$ we have $\mu(\mathcal{U})>0$.
\end{theorem}

\begin{theorem}[{\cite[Theorem 2]{DJ2018}}]\label{staest}
    Assume that $a\in C_c^\infty(T^*M)$ and $a|_{S^*M}\not\equiv 0$. Then, there exist constants $C(a),\, h_0>0$ depending on $M$ and $a$ such  that for $0<h<h_0$ and all $u\in H^2(M)$ \begin{equation}\label{obfree}
        \|u\|_{L^2}\leq C(a)\| \mathrm{Op}_h(a) u \|_{L^2} + C(a)\frac{\log(1/h)}{h} \| (-h^2\Delta-1) u \|_{L^2}.
    \end{equation}
\end{theorem}

It is obvious that Theorem \ref{fullsupp} can be deduced from Theorem \ref{staest}. To use the full-support property, we need to expand Theorem \ref{staest} to case of $-\Delta+V$ for $V\in L^{\frac{p}{p-2}},\, 2\leq p<4$.

Before expanding the theorem, we state the following important conclusion, the invariance of the semiclassical measure, which has been proved in \cite[Theorem 5.4]{SemianaZM} first. In the following, we provide a brief proof to facilitate the reader's understanding in our setting.
\begin{theorem}\label{Liederiv}
    For a family of $u_n$ that satisfies $\|u_n\|_{L^2}=1$ and $(P^n_V-1)u_n=o(h_n)$, $\mu$ is the semiclassical defect measure corresponding to $u_n$, where $h_n\to 0, \ \text{as}\ n\to \infty$. We have 
    \begin{equation}
        \mathcal{L}_{H_p}\mu=0,
    \end{equation}where $H_p$ represents the Hamiltonian vector field generated by the homogeneous geodesic flow $\Phi_t$.
\end{theorem}

\begin{proof}

We first claim that
\begin{equation}\label{claim}
    \|u_n\|_{L^p}\lesssim O(h_n^{-2\delta(p)-\epsilon}).
\end{equation} 
By the assumption, we have
$$
(P_V^n-1-i\eta)u_n=o(h_n)-i\eta u_n:=f_n-i\eta u_n.
$$
Using Proposition~\ref{resolwithV}, we obtain
\begin{equation}
    \begin{aligned}
        \|u_n\|_{L^p}\lesssim &\; h_n^{-1-2\delta(p)-\epsilon}\|f_n-i\eta u_n\|_{L^{p'}}\lesssim h_n^{-1-2\delta(p)-\epsilon}(o(h_n)+\eta\|u_n\|_{L^2}) \\ 
        \lesssim&\; O(h_n^{-2\delta(p)-\epsilon}),
    \end{aligned}
\end{equation}
where, in the last inequality, we use the fact that the preceding estimate holds for all $\eta\gtrsim h_n$. Hence we may simply take $\eta=Ch_n$.

Next, we prove the following claim. To show that $\mathcal{L}_{H_p}\mu=0$, it is sufficient to prove the average of $\frac{i}{h_n}[P_V^n-1,A]$ vanishes, that is, $\lim\limits_{n\to \infty}\langle \frac{i}{h_n}[P_V^n-1,A]u_n,u_n\rangle=0$, where $A=\mathrm{Op}_h(a)$ for some $a\in C_c^\infty(T^*M)$ with $a|_{S^*M}\not\equiv 0$. We first demonstrate that the potential $V$ is negligible in the semiclassical limit when considered in the mean, and hence the principal symbol of $P_V-1$ is given by $|\xi|^2-1$. Indeed, since $1-4\delta(p)>0$ for all $2\leq p<4$,
\begin{equation}
    \begin{aligned}
       \left|\langle \frac{i}{h_n}[h_n^2V,A]u_n,u_n\rangle\right|= & \;\left|h_n\int V(Au_n\overline{u_n}-u_n\overline{A^*u_n}) \right|\\
        \lesssim&\; h_n\|V\|_{L^{\frac{p}{p-2}}}(\|Au_n\|_{L^p}\|\overline{u_n}\|_{L^p}+\|u_n\|_{L^p}\|A^*u_n\|_{L^p}) \\
        \lesssim & \;  h_n\|V\|_{L^\frac{p}{p-2}}\|u_n\|^2_{L^p}  \lesssim h_n^{1-4\delta(p)-2\epsilon}\|V\|_{L^\frac{p}{p-2}}\to 0,
    \end{aligned}
\end{equation}
as $n\to \infty.$

Setting $p(\xi)=|\xi|^2-1$, we have 
$$
\frac{i}{h_n}[P_V^n-1,A]= \mathrm{Op}_{h_n}(\{p,a\})+o(1).
$$
By the definition of the semiclassical defect measure, we have \begin{equation}\label{Hpa=0}
    0=\lim_{n\to \infty}\langle \frac{i}{h_n}[P_V^n-1,A]u_n,u_n\rangle=\int_{T^*M} \{p,a\}d\mu=\int_{T^*M} H_pa(x,\xi)d\mu,
\end{equation} 
where the last equality follows from the definition of the Poisson bracket. 

We now prove that
\begin{equation}
    \mathcal{L}_{H_p}\mu=0.
\end{equation}

By \cite[Page 105]{SemianaZM} and \eqref{Hpa=0}, the measure $\mu$ is invariant along the homogeneous geodesic flow $\Phi_t$. Therefore, by the definition of Lie derivation we have \begin{equation}
    \mathcal{L}_{H_p}\mu:=\frac{d}{dt}\bigg|_{t=0}(\Phi_t^*\mu)=\frac{d}{dt}\bigg|_{t=0}\mu=0.
\end{equation}

It remains to prove that $\lim\limits_{n\to \infty}\langle \frac{i}{h_n}[P_V^n-1,A]u_n,u_n\rangle=0$. Indeed,
\begin{equation}
    \begin{aligned}
        \left|\langle \frac{i}{h_n}[P_V^n-1,A]u_n,u_n\rangle\right|&= \left|\frac{i}{h_n}\Big(\langle Au_n,(P_V^n-1)u_n\rangle-\langle (P_V^n-1)u_n,A^*u_n\rangle\Big) \right|\\
        &\lesssim  \frac{1}{h_n}(\|Au_n\|_{L^2}\|(P_V^n-1)u_n\|_{L^2}+\|A^*u_n\|_{L^2}\|(P_V^n-1)u_n\|_{L^2}) \\
        &\lesssim\frac{1}{h_n}o(h_n)\to 0,
    \end{aligned}
\end{equation}
as $n\to \infty$. This completes the proof.

\end{proof}

\begin{remark}
    The main conclusion and the proof process of the lemma indicate that the semiclassical defect measure defined by the eigenfunction of $-\Delta+V$ is an invariant measure along the free geodesic flow and as the potential $V$ is a negligible tiny perturbation, we can simplify the situation of $-h^2\Delta+h^2V(x)-1$ to $-h^2\Delta-1$ which has been proved in \cite{DJ2018} and we can obtain the expansive observability inequality Theorem~\ref{staobser}. And it is worth to note that the lemma merely provides a physical context and explains the rationality of our proof. We do not use the lemma directly in the proof.
\end{remark}

\begin{remark}
    In \cite[Page 3]{DJ2018}, they set $\|(-h_j^2\Delta-1)u_j\|_{L^2}=o(h_j/\log(1/h_j))\|$ to proceed Theorem \ref{staest} to Theorem \ref{fullsupp}. However, in the proof of the above theorem, we only need to set $(-h_n^2\Delta-1)u_n\sim(P_V^n-1)u_n=o(h_n)$. For the difference, we say that to prove the invariance of the measure, the assumption $o(h)$ is enough rather than the stronger assumption $o(h/\log(1/h))$. If we want to prove the full-support property, assumption $o(h/\log(1/h))$ is essential. From a physical perspective, we prove the invariance of the measure for waves with a wave velocity of 1 over a period of time $O(1)$ along the geodesic. The full-support property and the ergodicity of the geodesic must require the Ehrenfest time $O(\log(1/h))$, see \cite[Section 4.3]{DJ2018}. Thus, we have clearly explained the differences in the assumptions involved.
\end{remark}

\begin{proof}[Proof of Theorem~\ref{staobser}] 

For $\chi\in C_c^\infty([1/2,2])$, we only need to consider $\chi(P_V)u$. As for $(1-\chi(P_V))u$, we have $$\|(1-\chi(P_V))u\|_{L^2}\lesssim C\|(P_V-1)u\|_{L^2},$$which can be directly controlled.

Then, we start from the observability inequality \eqref{obfree} and denote $\chi(P_V)u$ as $u$ For simplicity. \begin{equation}
    \begin{aligned}
        \|u\|_{L^2}&\leq C(a)\| \mathrm{Op}_h(a) u \|_{L^2} + C(a)\frac{\log(1/h)}{h} \| (-h^2\Delta-1) u \|_{L^2} \\
        &=C(a)\| \mathrm{Op}_h(a) u \|_{L^2} + C(a)\frac{\log(1/h)}{h} \| (P_V-1) u -h^2Vu\|_{L^2} \\
        &\leq C\| \mathrm{Op}_h(a) u \|_{L^2} + C\frac{\log(1/h)}{h} \| (P_V-1) u \|_{L^2}+C\frac{\log(1/h)}{h} \|h^2Vu\|_{L^2} \\
        &\leq C\| \mathrm{Op}_h(a) u \|_{L^2} + C\frac{\log(1/h)}{h} \| (P_V-1) u \|_{L^2}+Ch\log(1/h) \|V\|_{L^{\frac{p}{p-2}}}\|u\|_{L^{\frac{2p}{4-p}}}.
    \end{aligned}
\end{equation}Due to $2\leq p<4$, using Sogge's spectral clusters estimate Theorem \ref{sogge} and Theorem \ref{sogge_disturb}, we have \begin{equation}
    \begin{aligned}
        Ch\log(1/h) \|V\|_{L^{\frac{p}{p-2}}}\|u\|_{L^{\frac{2p}{4-p}}}&\lesssim Ch\log(1/h) \|V\|_{L^{\frac{p}{p-2}}}(\sum_{\lambda\in [1/2h,2/h]}\|\Pi_{[\lambda,\lambda+1]}u\|_{L^{\frac{2p}{4-p}}})\\
        &\lesssim Ch^{1/2-\delta(\frac{2p}{4-p})}\log(1/h) \|V\|_{L^{\frac{p}{p-2}}}(\sum_{\lambda\in [1/2h,2/h]}\|\Pi_{[\lambda,\lambda+1]}u\|^2_{L^{2}})^{1/2}\\
        &\lesssim Ch^{1/2-\delta(\frac{2p}{4-p})}\log(1/h) \|V\|_{L^{\frac{p}{p-2}}}\|u\|_{L^2}.
    \end{aligned}
\end{equation}Because $\delta(\frac{2p}{4-p})<\frac{1}{2}$, for any $2\leq p<4$, we can select the parameter $h$ small enough such that $Ch^{1/2-\delta(\frac{2p}{4-p})}\log(1/h) \|V\|_{L^{\frac{p}{p-2}}}<\frac{1}{2}$. Thus, we complete the proof.

\end{proof}
Following from the conclusion of Theorem~\ref{staobser}, we have that the defect measure generated by the eigenfunctions of $-\Delta+V$ is full-support, that is $\supp \mu=S^*M$.

\begin{remark}
    It is worth noting that by using the method of \cite{DJ2018}, one can obtain more general conclusions of Theorem~\ref{staobser}. However the extension is not obvious, since in \cite[Page 6]{DJ2018}, the stable and unstable horocyclic vector fields for $-\Delta+V$ do not necessarily have such good properties and we need more precise geometric analysis or other innovative methods to overcome the difficulty.
\end{remark}

\subsection{Proof of observability inequality}

With the above preparation, we now prove the following spectrally localized  semiclassical Strichartz estimate.
\begin{theorem}\label{semistrichartz}
    Let $\chi_1\in C_c^\infty((\frac{1}{2},2))$ be a smooth cutoff function. For any $V\in L^\frac{p}{p-2}$ and $u_0\in L^2$, the solution to \begin{equation}
        (ih\partial_t+h^2\Delta-h^2V(x))u=0,\ u|_{t=0}=\chi_1(P_V)u_0
    \end{equation}satisfies \begin{equation}
    \|u\|_{L^p_x(M;L^2_t([0,T]))}^2\lesssim h^{-4\delta(p)-2\epsilon}\|u_0\|_{L^{2}}^2.
\end{equation}
\end{theorem}

\begin{proof}
First, by taking the inverse semiclassical Fourier transform with respect to $t$, we have 
$$
\mathcal{F}_h^{-1}[ih\partial_tw]=\tau\check{w}_h
$$
and 
\begin{equation}\label{normfourier}
    \|\check{w}_h\|_{L^2_\tau}^2=2\pi h\|w\|_{L^2_t}^2.
\end{equation}
Let $w(t)=\chi(t)u(t)=\chi(t)e^{-itP_V/h}\chi_1(P_V)u_0$, where $\chi(t)\in C_c^\infty(\mathbb{R})$ is supported in $[0,T]$ and satisfies $\chi(t)\equiv 1\ \text{on}\ [T/10,9T/10]$. Then we consider
\begin{equation}
    (ih\partial_t-P_V)w=f:=ih\chi'(t)u(t).
\end{equation}Setting $v=e^{-\frac{\eta t}{h}}w$, we obtain\begin{equation}
    (ih\partial_t-P_V+i\eta)v=e^{-\frac{\eta t}{h}}f:=F.
\end{equation}
Taking the inverse semiclassical Fourier transform with respect to $t$, we get
\begin{equation}\label{1.28}
    (-P_V+\tau+i\eta)\check{v}_h=\check{F}_h.
\end{equation}
According to Theorem~\ref{fullsupp} and the proof of Theorem~\ref{staobser}, the proof of Theorem~\ref{Liederiv} and the assumption, it is enough to consider the problem near $S^*M$, where $\operatorname{spec}(P_V)\approx 1$. Applying the non-stationary phase argument to $\check{f}_h$, we obtain 
\begin{equation}
    \begin{aligned}
        \check{f}_h(\tau,x)&=\int e^{it\tau/h}ih\chi'(t)e^{-itP_V/h}\chi_1(P_V)u_0(x)dt =ih\check{\chi}_h(\tau-P_V)\chi_1(P_V)u_0(x)\\
        &=ih\left(\frac{h}{i(\tau-P_V)}\right)^N\int e^{it(\tau-P_V)/h}\chi^{(N+1)}(t)dt\,\chi_1(P_V)u_0(x).
    \end{aligned}
\end{equation}When $\tau$ lies far from the spectrum of $P_V$, this term gives a negligible contribution $\mathcal{O}(h^\infty)$ to the Strichartz estimate. Combined with the fact that Proposition~\ref{resolwithV} holds uniformly for $\tau$ in the resonant regime, which means that we can prove Proposition~\ref{resolwithV} for 1 replaced by $\tau\in [1/2,2]$, up to a negligible error,  we may write
$$
(-P_V+1+i\eta)\check{v}_h=\check{F}_h
$$ 
instead of \eqref{1.28}.



Using Proposition~\ref{resolwithV} and \eqref{normfourier}, we obtain \begin{equation}\label{Strichartz}
    \begin{aligned}
        \|\check{w}_h\|_{L^p_x(M;L^2_\tau(\mathbb{R}))}^2\sim& \|\check{v}_h\|_{L^p_x(M;L^2_\tau(\mathbb{R}))}^2\lesssim\|\check{v}_h\|_{L^2_\tau(\mathbb{R};L^p_x(M))}^2\lesssim h^{-2-4\delta(p)-2\epsilon}\|\check{F}_h\|_{L^2_\tau(\mathbb{R};L^{p'}_x(M))}^2\\
        \lesssim&\; h^{-2-4\delta(p)-2\epsilon}\|\check{F}_h\|_{L^{p'}_x(M;L^2_\tau(\mathbb{R}))}^2\lesssim h^{-2-4\delta(p)-2\epsilon}h\|F\|_{L^{p'}_x(M;L^2_t(\mathbb{R}))}^2\\
        \lesssim &\; h^{-1-4\delta(p)-2\epsilon}\|f\|_{L^{2}_x(M;L^2_t(\mathbb{R}))}^2\lesssim h^{1-4\delta(p)-2\epsilon}\|u\|_{L^{2}_x(M;L^2_t([0,T]))}^2\\
        \lesssim&\; h^{1-4\delta(p)-2\epsilon}\|u_0\|_{L^{2}}^2.
    \end{aligned}
\end{equation}
Since 
$$
\|\check{w}_h\|_{L^p_x(M;L^2_\tau(\mathbb{R}))}^2\simeq h\|w\|_{L^p_x(M;L^2_t(\mathbb{R}))}^2\simeq h\|u\|_{L^p_x(M;L^2_t([0,T]))}^2,
$$ 
combining this with \eqref{Strichartz} completes the proof.
\end{proof}

\begin{remark}
    This lemma will not be directly used during the proof. But we can use this lemma instead of the claim \eqref{claim} to prove that the potential is negligible in the context of space and time and further prove the invariance of the measure, which provides sufficient physical justification for our proof..
\end{remark}

Next, we prove the semiclassical observability result.
\begin{proposition}\label{semiob}
    Let $\psi\in C_c^\infty(\mathbb{R};[0,1])$ not identically zero, then for all $u_0\in L^2(M)$, we have \begin{equation}
        \|\chi_1(P_V)u_0\|^2_{L^2(M)}\lesssim \int_\mathbb{R}\|\psi(t)e^{-it(-\Delta+V)}\chi_1(P_V)u_0\|^2_{L^2(\Omega)}dt.
    \end{equation}In particular, for any $T>0$, we have \begin{equation}
        \|\chi_1(P_V)u_0\|^2_{L^2(M)}\lesssim \int_0^T\|e^{-it(-\Delta+V)}\chi_1(P_V)u_0\|^2_{L^2(\Omega)}dt.
    \end{equation}
\end{proposition}
\begin{proof}
    We set $v(t)=e^{-ith(-\Delta+V)}\chi_1(P_V)u_0$ and write $w(t)=\psi(ht)v(t)$. It is obvious that $v(t)$ satisfies the Schr\"odinger equation $(ih\partial_t-P_V)v=0$ and $w(t)$ solves $$(ih\partial_t-P_V)w=ih^2\psi'(ht)v.$$

    Taking the inverse semiclassical Fourier transform, we can get $$(P_V-\tau)\check{w}_h(\tau)=-ih^2\mathcal{F}_h^{-1}[\psi'(ht)v](\tau).$$

    Choosing $a\in C_c^\infty$ to be supported in $\{(x,\xi):x\in\Omega\}$ and $\chi\in C_c^\infty(\Omega;[0,1])$ to be regarded as a function on $T^*M$, such that $\chi\equiv 1$ on the neighborhood of $\supp a$, we have $\mathrm{Op}_h(a) \check{w}_h=\mathrm{Op}_h(a) \chi \check{w}_h+\mathrm{Op}_h(a) (1-\chi)\check{w}_h$ where $\mathrm{Op}_h(a) (1-\chi)=\mathcal{O}_{L^2\to L^2}(h^\infty)$. Then, we have $$\|\mathrm{Op}_h(a) \check{w}_h\|_{L^2(M)}\leq \|\check{w}_h(\tau)\|_{L^2(\Omega)}.$$

    Thus, following the same strategy as \cite[Proposition 2.1]{JLcon}, considering $\tau\in [1/2,2]$, using from Theorem~\ref{staobser}, we have \begin{equation}
        \|\check{w}_h(\tau)\|_{L^2(M)}\lesssim\|\check{w}_h(\tau)\|_{L^2(\Omega)}+hlog(1/h)\|\mathcal{F}_h^{-1}[\psi'(ht)v](\tau)\|_{L^2(M)}.
    \end{equation}
    
    As for $\tau\not\in[\frac{1}{2},2]$, since \begin{equation}
        \check{w}_h(\tau)=\int_\mathbb{R}e^{-it(P_V-\tau)/h}\psi(ht)\chi_1(P_V)u_0dt
    \end{equation} and integration by parts, we get \begin{equation}
        \|\check{w}_h(\tau)\|_{L^2(M)}=\mathcal{O}_{L^2\to L^2}((h\langle\tau\rangle^{-1})^\infty)\|\chi_1(P_V)u_0\|_{L^2(M)}.
    \end{equation}

    Therefore, we have \begin{equation}
        \begin{aligned}
            \|\check{w}_h(\tau)\|_{L^2(M)}\lesssim&\|\check{w}_h(\tau)\|_{L^2(\Omega)}+hlog(1/h)\|\mathcal{F}_h^{-1}[\psi'(ht)v](\tau)\|_{L^2(M)} \\
            &+\mathcal{O}_{L^2\to L^2}((h\langle\tau\rangle^{-1})^\infty)\|\chi_1(P_V)u_0\|_{L^2(M)}
        \end{aligned}
    \end{equation} and \begin{equation}
        \begin{aligned}
            \|\check{w}_h(\tau)\|_{L^2(\mathbb{R}_\tau,L^2(M))}\lesssim&\|\check{w}_h(\tau)\|_{L^2(\mathbb{R}_\tau,L^2(\Omega))}+hlog(1/h)\|\mathcal{F}_h^{-1}[\psi'(ht)v](\tau)\|_{L^2(\mathbb{R}_\tau,L^2(M))} \\
            &+\mathcal{O}_{L^2\to L^2}(h^\infty)\|\chi_1(P_V)u_0\|_{L^2(M)}.
        \end{aligned}
    \end{equation}
    Using the definition of the inverse semiclassical Fourier transform, we obtain\begin{equation}
        \begin{aligned}
            \|w_h(t)\|_{L^2(\mathbb{R},L^2(M))}\lesssim&\|w_h(t)\|_{L^2(\mathbb{R},L^2(\Omega))}+hlog(1/h)\|\psi'(ht)v\|_{L^2(\mathbb{R},L^2(M))} \\
            &+\mathcal{O}_{L^2\to L^2}(h^\infty)\|\chi_1(P_V)u_0\|_{L^2(M)}
        \end{aligned}
    \end{equation}and further\begin{equation}
        \begin{aligned}
            h^{-1}\|\psi(t)\|_{L^2(\mathbb{R})}^2\|\chi_1(P_V)u_0\|^2_{L^2(M)}\lesssim &h^{-1}\int_\mathbb{R}\|\psi(ht)e^{-iht(-\Delta+V)}\chi_1(P_V)u_0\|^2_{L^2(\Omega)}dht \\
            &+hlog(1/h)h^{-1}\|\psi'(t)\|_{L^2(\mathbb{R})}^2\|\chi_1(P_V)u_0\|^2_{L^2(M)} \\
            &+\mathcal{O}_{L^2\to L^2}(h^\infty)\|\chi_1(P_V)u_0\|_{L^2(M)}^2.
        \end{aligned}
    \end{equation}
    
    Setting $h$ small enough, we have finished the proof.

\end{proof}

\begin{lemma}[Spectral averaging lemma]\label{spectral_ave}
    Let $H_V:=-\Delta+V$ be a self-adjoint operator on $L^2$ and $P_{H_V}$ denote its projection-valued spectral measure, let $E\in \mathbb{R}$. Then, for every $u\in L^2$, we have \begin{equation}
        \lim_{R\to \infty}\|\frac{1}{2R}\int_{-R}^Re^{-it(H_V-E)}udt-P_{H_V}(\{E\})u\|_{L^2}=0.
    \end{equation}
\end{lemma}
\begin{proof}
    We set $$\mathcal{A}_Ru=\frac{1}{2R}\int_{-R}^Re^{-it(H_V-E)}udt$$ which is bounded with norm $\|\mathcal{A}_R\|\leq 1$. For $s\in\mathbb R$, we define $$m_R(s)=\frac{1}{2R}\int_{-R}^Re^{-it(s-E)}dt=\begin{cases}
        1,\qquad &s=E, \\
        \frac{\sin(R(s-E))}{R(s-E)}, \, &s\not=E.
    \end{cases}$$ Thus, we have $\mathcal{A}_R=m_R(H_V)$. For any fixed $u\in L^2$ and $S\subset \mathbb{R}$, we introduce the finite positive measure $$\mu_u(\mathbb{R})=\|u\|^2_{L^2}, \quad \mu_u(S)=\|P_{H_V}(S)u\|^2_{L^2}.$$
    
    Therefore, we obtain \begin{equation}
        \begin{aligned}
            &\|\frac{1}{2R}\int_{-R}^Re^{-it(H_V-E)}udt-P_{H_V}(\{E\})u\|^2_{L^2}=\|m_R(H_V)u-P_{H_V}(\{E\})u\|^2_{L^2} \\
            =&\|(m_R-\mathbf{1}_{\{E\}})(H_V)u\|^2_{L^2}=\int_{\mathbb{R}}|m_R-\mathbf{1}_{\{E\}}|^2d\mu_u(s).
        \end{aligned}
    \end{equation}Since for any $s\in \mathbb{R}$, we have $m_R(s)\to \mathbf{1}_{\{E\}}(s)$ as $R\to \infty$ and $|m_R-\mathbf{1}_{\{E\}}|^2\leq 1$, using the dominated convergence theorem, we complete the proof.
\end{proof}

We are now ready to prove the main theorem of this paper.
\begin{proof}[Proof of Theorem~\ref{obser}]
We prove again by contradiction. Assume there exists a sequence of $\{u_n\}$ that satisfy $\|u_n(0)\|_{L^2}=1$ and $\int_0^T\|u_n(t,\cdot)\|_{L^2(\Omega)}^2dt\to 0$ where $u_n(t,x)=e^{-it(-\Delta+V)}u_n(0,x)$. Thus, we have $u_n(x)\rightharpoonup u(x)$ and $\int_0^T\|e^{-it(-\Delta+V)}u(0,x)\|_{L^2(\Omega)}^2dt=0$. Let operator $B$ be the restriction operator $B:L^2(M)\to L^2(\Omega)$ and $F_g(z)=\langle e^{-izH_V}f,B^*g\rangle$ for any test function $g\in L^2(\Omega)$. We have $$Be^{-itH_V}u=0\quad \text{for almost every } t\in (0,T).$$ Since $F_g(z)$ is holomorphic on $Im z<0$ and continuous onto the real axis, using the Schwarz reflection principle to $F_g(z)$ and the extension property, we get $$Be^{-itH_V}u=0\quad \text{for all } t\in \mathbb{R}.$$Using Lemma~\ref{spectral_ave} directly, we have that $$v_R:=\mathcal{A}_Ru=\frac{1}{2R}\int_{-R}^Re^{-it(H_V-E)}udt=\frac{1}{2R}\int_{-R}^Re^{itE}e^{-itH_V}udt$$converges strongly to $\Pi_Eu$, for any $E\in \mathbb{R}$. Then we can get $$Bv_R=\frac{1}{2R}\int_{-R}^Re^{itE}Be^{-itH_V}udt=0,\quad\text{for every}\, R>0$$ and $$B\Pi_Eu:=Bv=0,$$which means $v|_{\Omega}=0$.

Since the following equation \begin{equation}
    (-\Delta+V-E)v=0,\quad v|_{\Omega}=0,
\end{equation}and the Strong Unique Continuation Property Lemma~\ref{lem:ucp}, we obtain $v=0$. Therefore, we can get obviously $\Pi_Eu=0$, for any distinct eigenvalue $E$. Combined with the completeness of the eigenspaces, we have $u=0$ and $u_n(x)\rightharpoonup 0$.

The proof proceeds by considering two cases: the high frequency term $h_n^{-1}\to \infty$ and the low frequency term $h_n^{-1}\leq R<\infty$, where $R$ is a constant large enough. For each $u_n(x)$, we decompose it $u_n=\Pi_{\leq R}u_n+\Pi_{> R}u_n:=u_{n,far}+u_{n,near}$, where $u_{n,near}$ represents the low frequency term and $u_{n,far}$ represents the high frequency term.

For the low frequency term $\{u_{n,near}\}$, since the spectrum is a discrete spectrum and the projection operator $\Pi_{\leq R}$ is of finite rank where $\Pi_{\leq R}$ projects onto the spectral subspace $\mathcal{H}_{R}:=span\{\phi_j, \lambda_j\leq R\}$ which is finite-dimensional, combined with $\Pi_{\leq R}u_n(x)\rightharpoonup 0$, we have that $\Pi_{\leq R}u_n(x)\to 0$.

For the high frequency term $\{u_{n,far}\}$, since the low frequency term converge strongly to zero, we select a subsequence such that $\|u_{n,far}\|_{L^2}\geq \frac{1}{2}$. Exactly, we have $\lim_{n\to\infty}\|u_{n,far}\|_{L^2}=1$. By suitable scaling, we can reset $\|u_{n,far}\|_{L^2}=1$ and $\int_0^T\|e^{-itH_V}u_{n,far}(t,\cdot)\|_{L^2(\Omega)}^2dt\to 0$, as $n\to \infty$. We further decompose the high frequency term by the partition of unity $$\phi_0^2(s)+\sum_{j=1}^\infty\phi^2(2^{-2j}s)=I$$ where $\phi_0\in C_c^\infty([-2,2]),\, \phi\in C_c^\infty([1/4,4])$ and we extended $\phi$ to an even function. By the decomposition, we define $$w_{n,j}=\phi(2^{-2j}(-\Delta+V))u_{n,far}\quad \text{for all}\ n\ \text{and}\ j\geq J,$$where $J=\max\{j,2^j\leq R\}$ and $\lim_{n\to\infty}\sum_{j=J}^\infty\|w_{n,j}\|_{L^2}^2=\lim_{n\to\infty}\|u_{n,far}\|_{L^2}^2=1$. Using Proposition~\ref{semiob}, for any $n,j$, we have \begin{equation}
    \|w_{n,j}\|^2_{L^2(M)}\lesssim \int_\mathbb{R}\|\psi(t) e^{-itH_V}w_{n,j}\|^2_{L^2(\Omega)}dt.
\end{equation}

Since the equation $(D_t+H_V)e^{-itH_V}=0$, for $D_t=-i\partial_t$, we get $$\psi(t) e^{-itH_V}\phi(2^{-2j}H_V)=\psi(t)\phi(2^{-2j}H_V) e^{-itH_V}=\psi(t)\phi(-2^{-2j}D_t) e^{-itH_V}=\psi(t)\phi(2^{-2j}D_t) e^{-itH_V}$$ and $$\psi(t)\phi(2^{-2j}D_t)=\psi(t)\phi(2^{-2j}D_t)\tilde{\psi}(t)+\psi(t)[\tilde{\psi}(t),\phi(2^{-2j}D_t)]$$where $\tilde{\psi}\in C_c^\infty((0,T);[0,1])$ is a smooth function such that $\tilde{\psi}\equiv1$ on $\supp \psi$ and for any $N>0$, we have $$\partial^\alpha\left[\psi(t)[\tilde{\psi}(t),\phi(2^{-2j}D_t)]\right]=\mathcal{O}(2^{-jN}\langle t\rangle^{-N}\langle D_t\rangle^{-N}).$$

Therefore, we have \begin{equation}
    \begin{aligned}
        \sum_{j=J}^\infty\|w_{n,j}\|_{L^2}^2\lesssim&\sum_{j=J}^\infty\int_\mathbb{R}\|\psi(t) e^{-itH_V}w_{n,j}\|^2_{L^2(\Omega)}dt \\
        \lesssim&\sum_{j=J}^\infty\int_\mathbb{R}\|\phi(2^{-2j}D_t)\tilde{\psi}(t) e^{-itH_V}u_{n,far}\|^2_{L^2(\Omega)}dt \\
        &+\sum_{j=J}^\infty2^{-jN}\int_\mathbb{R}\langle t\rangle^{-N}\|e^{-itH_V}u_{n,far}\|^2_{L^2(\Omega)}dt \\
        \lesssim& \int_0^T\| e^{-itH_V}u_{n,far}\|^2_{L^2(\Omega)}dt+2^{-JN}\|u_{n,far}\|^2_{L^2(M)}.
    \end{aligned}
\end{equation}

As $n$ goes to infinity and we select the constant $J$, $N$ and $R$ large enough, we obtain a contradiction and finish the proof.
\end{proof}

\noindent{\textbf{Acknowledgements}: } T. Z. was supported by Natural Science Foundation of Zhejiang Province (Grant No. LDQ23A010001). 

	\vspace{2mm}
	
\noindent \textbf{Conflict of interest.} On behalf of all authors, the corresponding author states that there is no conflict of interest.
	\vspace{2mm}

\noindent \textbf{Data availability statement.}
	Data sharing is not applicable to this article as no data sets were generated or analyzed during the current study.


\end{document}